\documentclass[11pt]{amsart}
\usepackage[top=3.5cm,bottom=3.5cm,left=3.8cm,right=3.8cm]{geometry}               
\usepackage{amsmath,amscd,amssymb,amsthm}
\usepackage{mathtools}
\usepackage{enumerate}
\usepackage{setspace}
\usepackage{mathrsfs}
\usepackage{color}
\usepackage[numbers]{natbib}
\usepackage[hidelinks]{hyperref}

\DeclarePairedDelimiter{\floor}{\lfloor}{\rfloor}

\DeclareMathOperator{\Gal}{Gal}
\DeclareMathOperator{\Aut}{Aut}
\DeclareMathOperator{\disc}{disc}
\DeclareMathOperator{\car}{char}
\DeclareMathOperator{\rad}{rad}

\def\vF{\mathbb{F}}
\def\F{\mathbb{F}}
\def\vN{\mathbb{N}}
\def\N{\mathbb{N}}

\def\Z{\mathbb{Z}}

\def\Q{\mathbb{Q}}

\def\O{\mathcal{O}}
\def\T{\mathcal{T}}
\def\p{\mathfrak{p}}
\def\q{\mathfrak{q}}

\theoremstyle{definition}
\newtheorem{definition}{Definition}[section]
\newtheorem{remark}[definition]{Remark}

\theoremstyle{plain}
\newtheorem{theorem}[definition]{Theorem}
\newtheorem{corollary}[definition]{Corollary}
\newtheorem{lemma}[definition]{Lemma}

\author[A. Ferraguti]{Andrea Ferraguti}
\address{Max Planck Institute for Mathematics\\
Vivatsgasse 7\\
53111 Bonn, Germany
}
\email{ferra@mpim-bonn.mpg.de}

\author[G. Micheli]{Giacomo Micheli}
\address{Mathematical Institute\\
EPFL\\
Station 8\\ 
Lausanne 1015, Switzerland
}
\email{giacomo.micheli@epfl.ch}

\title[Equivariant isomorphisms for arboreal Galois Representations]{An equivariant isomorphism theorem for mod $\p$\\ reductions of arboreal Galois representations}

\thanks{The first author is grateful to the Max Planck Institute for Mathematics in Bonn for its hospitality and financial support. The second author was supported by the Swiss National Science Foundation grant number 171249. The first author would also like to thank the EPFL, where most of the ideas of this paper have been generated, for the hospitality and financial support.}


\begin{document}

\begin{abstract}
Let $\phi$ be a quadratic, monic polynomial with coefficients in $\O_{F,D}[t]$, where $\O_{F,D}$ is a localization of a number ring $\O_F$. In this paper, we first prove that if $\phi$ is non-square and non-isotrivial, then there exists an absolute, effective constant $N_\phi$ with the following property: for all primes $\p\subseteq\O_{F,D}$ such that the reduced polynomial $\phi_\p\in (\O_{F,D}/\p)[t][x]$ is non-square and non-isotrivial, the squarefree Zsigmondy set of $\phi_\p$ is bounded by $N_\phi$. Using this result, we prove that if $\phi$ is non-isotrivial and geometrically stable then outside a finite, effective set of primes of $\O_{F,D}$ the geometric part of the arboreal representation of $\phi_\p$ is isomorphic to that of $\phi$. As an application of our results we prove R.\ Jones' conjecture on the arboreal Galois representation attached to the polynomial $x^2+t$.
\end{abstract}

\maketitle
\section{Introduction}
Let $k$ be a field, let $t$ be transcendental over $k$ and let $\phi\in k(t)(x)$ be a rational function of degree $d$. Suppose that $\phi^{(n)}$ is separable for every $n$, where $\phi^{(n)}$ is the $n$-th iterate. We fix separable closures $k^{\text{sep}}$ of $k$ and $k(t)^{\text{sep}}$ of $k(t)$, and we let $G_{k(t)}\coloneqq \Gal(k(t)^{\text{sep}}/k(t))$. One can associate to $\phi$ an infinite, regular, $d$-ary tree $\T_\infty(\phi)$ rooted in 0: the nodes at distance $n$ from the root are labeled by the roots of $\phi^{(n)}$ and a node $\alpha$ at level $n+1$ is connected to a node $\beta$ at level $n$ if and only if $\phi(\alpha)=\beta$. An \emph{automorphism} of $\T_\infty(\phi)$ is a bijection $\sigma$ of the set of nodes such that $\alpha$ is connected to $\beta$ if and only if $\sigma(\alpha)$ is connected to $\sigma(\beta)$; the full automorphism group $\Aut(\T_\infty(\phi))$ coincides with $\varprojlim_n\Aut(\T_n(\phi))$, where $\T_n(\phi)$ is the set of nodes at distance $\leq n$ from the root. It follows that $\Aut(\T_\infty(\phi))$ carries a natural profinite topology.

The group $G_{k(t)}$ acts on $\T_\infty(\phi)$ through tree automorphisms; the corresponding continuous group homomorphism $\rho_{\phi}\colon G_{k(t)}\to \Aut(\T_{\infty}(\phi))$ is called the \emph{arboreal Galois representation} attached to $\phi$. The study of such maps over various ground fields is a central topic in modern arithmetic dynamics, as witnessed by the many papers on the topic such as \citep{silverman2,boston,ferraguti1,ferraguti2,hindes1,hindes3,hindes2,
hindjon,jones1,jones2}. In particular, in the setting described above the coefficients of $\phi$ belong to the rational function field over $k$, and it is therefore important to distinguish the image of $\rho_\phi$ from its geometric part, as explained below. Boston and Jones \cite{boston} and Hindes \cite{hindes1} studied in detail some of the features of this framework.

\begin{definition}
 Let $k$ and $\phi$ be as above.
 \begin{enumerate}
 \item The \emph{geometric part} of $\rho_\phi$ is the image of the subgroup $\Gal(k(t)^{\text{sep}}/k^{\text{sep}}(t))\subseteq G_{k(t)}$ via $\rho_\phi$. We denote it by $G_\infty^{\text{geo}}(\phi)$, and it is identified, via $\rho_\phi$, with $\varprojlim_n\Gal(\phi^{(n)}/k^{\text{sep}}(t))$.
  \item We say that $\rho_\phi$ has \emph{geometrically finite index} if $[\Aut(\T_{\infty}(\phi)):G_\infty^{\text{geo}}(\phi)]<\infty$. In particular, we say that $\rho_\phi$ is \emph{geometrically surjective} if $G_\infty^{\text{geo}}(\phi)=\Aut(\T_{\infty}(\phi))$.
 \end{enumerate}
\end{definition}

In this paper, we will focus on monic, quadratic polynomials $\phi=(x-\gamma)^2-c$ with coefficients $\gamma,c$ in the polynomial ring $k[t]$, where $k$ is a field of characteristic different from 2. The arithmetic of the arboreal representations attached to such polynomials is deeply related to that of their \emph{adjusted post-critical orbit}, which is the sequence defined by: $c_1=-\phi(\gamma)$, $c_n=\phi^{(n)}(\gamma)$ for $n\geq 2$. In particular, linear dependence relations among the $c_n$'s in the $\F_2$-vector space $k(t)^{\times}/{k(t)^{\times}}^2$ are of crucial importance, as shown for example in \citep{hindes1,jones2,stoll}. For this reason, one is interested in studying \emph{primitive prime divisors} for $c_n$, namely irreducible elements $g\in k[t]$ such that $g\mid c_n$ and $g\nmid c_i$ for every $i<n$ such that $c_i\neq 0$. We call a primitive prime divisor \emph{squarefree} if, in addition, $v_g(c_n)\equiv 1 \bmod 2$. The \emph{Zsigmondy set} of $\phi$ is the set $\mathcal Z(\phi)\coloneqq\{n\in\N\colon c_n\mbox{ has no primitive prime divisors}\}$, see for example \citep{hindes3,silverman1} for more on the topic.
\begin{definition}
 The \emph{squarefree Zsigmondy set} of $\phi$ is defined by:
 $$\mathcal Z_s(\phi)\coloneqq\{n\in\N\colon c_n\mbox{ has no squarefree primitive prime divisors}\}.$$
\end{definition}

Notice that if $\phi\in k[t][x]$, then $\mathcal Z_s(\phi)=\mathcal Z_s^{geo}(\phi)$, where $\mathcal Z_s^{geo}(\phi)$ is the squarefree Zsigmondy set of $\phi\in k^{sep}[t][x]$. For this reason, there is no ambiguity when talking about the Zsigmondy set of a quadratic polynomial in $k[t][x]$, when seen as an element of the bigger ring $k^{\text{sep}}[t][x]$.

From now on, we let $F$ be a number field with number ring $\O_F$. Let $D$ be a finite set of primes of $\O_F$ containing all primes dividing 2, and denote by $\O_{F,D}$ the localization of $\O_F$ at $D$ (i.e.\ we allow denominators which belong to some $\p\in D$). Let $\phi=(x-\gamma)^2-c_1$, where $\gamma,c_1\in \O_{F,D}[t]$. For every $\p\notin D$ the reduction $\phi_\p$ of $\phi$ modulo $\p$ yields a monic quadratic polynomial in $\F_\p[t][x]$, where $\F_\p=\O_F/\p$. Recall that a polynomial $\phi=(x-\gamma)^2-c_1\in k[t][x]$ is called \emph{isotrivial} if there exists $m\in k[t]$ such that $\phi(x-m)+m\in k[x]$. Equivalently, $\phi$ is isotrivial if $\deg(\gamma+c_1)=0$.

In \cite[Conjecture 6.7]{jones1} Jones conjectured that the arboreal Galois representation of $x^2+t\in k[t][x]$ is surjective for any field $k$ of characteristic different from $2$. When $k$ has characteristic 0 or 3 modulo 4, this conjecture was proven by Jones in \cite[Section 3.5]{jones2005galois} by adapting an argument of Stoll \cite[\textsection 2]{stoll} that was used by the author to show that the arboreal Galois representation attached to certain polynomials of the form $x^2+a$, with $a\in\Z$, is surjective. However, this is a highly ``ad hoc'' argument, and it fails in a fundamental way when the characteristic is 1 modulo 4.

In this paper we show that the aforementioned conjecture is an instance of a much more general fact, of arithmetic and geometric nature, concerning the squarefree Zsigmondy set attached to a quadratic polynomial $\phi\in\O_{F,D}[t][x]$ satisfying suitable hypotheses. More specifically, we will show how it is possible to compare the geometric part of the arboreal Galois representation attached to  $\phi$ and the geometric part of the arboreal Galois representation attached to the reduced polynomial $\phi_\p$. This can be viewed as an instance of ``arithmetic specialization'' of the representation. The key idea is to provide, given a non-isotrivial, non-square $\phi$, a uniform and \emph{effective} bound on the squarefree Zsigmondy set of $\phi_\p$, for all but finitely many $\p$'s. This will allow to show that the geometric part of the arboreal representation attached to $\phi$ does not change (in a suitable sense, cf.\ Definition \ref{equivariant}) after reducing $\phi$ modulo $\p$, for all primes outside an effective, finite set. In turn, an application of our results yields a complete proof of Jones' conjecture (cf.\ Theorem \ref{conjecture_proof}).

The first main result of the paper is the following theorem.
\begin{theorem}\label{main_thm}
 Suppose that $\phi=(x-\gamma)^2-c_1\in\O_{F,D}[t][x]$ is not isotrivial and $c_1\neq 0$. Then there exists an effective constant $N_\phi\in \N$ with the following property: let $\p$ be a prime of $\O_{F,D}$ such that $\phi_\p\in\F_\p[t][x]$ is not isotrivial and $c_1\not\equiv 0 \bmod \p$, and suppose that $n\in\mathcal Z_s(\phi_\p)$. Then $n\leq N_\phi$.
\end{theorem}

One of the key ingredients of the proof is the notion of \emph{dynamical inseparability degree} of a quadratic polynomial in positive characteristic (which we introduce with Definition \ref{def:insepdeg}) that allows to transfer methods for height bounds in characteristic zero to positive characteristic via a version of the ABC theorem for function fields. The dynamical inseparability degree of a non-isotrivial quadratic map is a way of measuring the degeneracy of the problem in positive characteristic by a non-negative integer, which is then used in the height bounds.

The constant $N_\phi$ mentioned in Theorem \ref{main_thm} can be easily made completely explicit: it just depends on the reduction of $\phi$ modulo a finite, effectively computable set of primes of $\O_{F,D}$. Notice that if $\phi$ is not isotrivial and $c_1\neq 0$, then $\phi_\p$ is not isotrivial and $c_1\not\equiv 0\bmod \p$ for all but finitely many primes $\p$. Finally, observe that the hypotheses of Theorem \ref{main_thm} are sharp: if $\phi$ is isotrivial then it is post-critically finite modulo $\p$ for every $\p$, and therefore $\mathcal Z_s(\phi_\p)$ is infinite. If $c_1=0$, then $\phi_\p$ is a square for every $\p$, and $\mathcal Z_s(\phi_\p)=\N$.

\vspace{3mm}

Next, we will use Theorem \ref{main_thm} in order to compare the geometric part of the arboreal representation attached to $\phi$ and the geometric part of the representation attached to the reduced polynomial $\phi_\p$. In doing so, we must take care of the following subtlety: the tree $\T_\infty(\phi)$ on which $G_\infty^{\text{geo}}(\phi)$ acts and the tree $\T_\infty(\phi_\p)$ on which $G_\infty^{\text{geo}}(\phi_\p)$ acts are isomorphic as trees, but they are not the same object. Therefore, in order to compare the two representations one needs to choose an identification between the trees. This motivates the following definition.

\begin{definition}\label{equivariant}
 Let $d\in \N$ and let $\T_\infty,\T'_\infty$ be infinite, regular, rooted $d$-ary trees. Let $n\in\N\cup\{\infty\}$ and let $\T_n,\T'_n$ be the trees truncated at level $n$. Let $G,G'$ be topological groups acting continuously on $\T_n,\T'_n$, respectively. An \emph{equivariant isomorphism} $(G,\T_n)\to (G',\T'_n)$ is a pair $(\varphi,\iota)$, where $\varphi\colon G\to G'$ is an isomorphism of topological groups and $\iota\colon \T_n\to\T_n'$ is a tree isomorphism such that for every $g\in G$ and every $t\in \T_n$ one has $\iota({}^{g}t)={}^{\varphi(g)}\iota(t)$. If there exists an equivariant isomorphism between the two pairs, we write $(G,\T_n)\cong_{\text{eq}} (G',\T'_n)$.
\end{definition}

Definition \ref{equivariant} aims to encapture the structure of the tree as a $G$-set. For example, the group $C_2\times C_2$ can act on the binary tree truncated at level 2 in two different ways: transitively or non-transitively. Of course in our analysis we want to consider these as two distinct objects.

We will make use of Theorem \ref{main_thm} to prove the following result. Recall that $\phi\in k[t][x]$ is called \emph{geometrically stable} if $\phi^{(n)}$ is irreducible in $k^{\text{sep}}[t][x]$ for every $n$.

\begin{theorem}\label{gal_rep}
  Let $\phi\in \O_{F,D}[t][x]$ be a quadratic, monic, geometrically stable and non-isotrivial polynomial. Then there exists a finite, effective set $S_\phi$ of primes of $\O_{F,D}$ such that $(G_\infty^{\text{geo}}(\phi_\p),\T_\infty(\phi_\p))\cong_{\text{eq}} (G_\infty^{\text{geo}}(\phi),\T_\infty(\phi))$ if and only if $\p\notin S_\phi$.
\end{theorem}

Moreover, the arboreal representation attached to a polynomial $\phi$ satisfying the hypotheses of Theorem \ref{gal_rep} has geometrically finite index, as shown in \cite{hindes1}. We will re-obtain this result in the course of our proof.

Notice that non-isotriviality is a necessary condition for the statement of Theorem \ref{gal_rep} to hold: if for example $\phi=(x-t)^2+t+1$, it is immediate to check that $\rho_\phi$ is geometrically surjective (cf.\ Theorem \ref{stoll_thm}); however, $\phi_p$ is post-critically finite for every prime $p$, and therefore the geometric part of $\rho_{\phi_p}$ has infinite index in $\Aut(\T_\infty(\phi_p))$ (see \cite[Theorem 3.1]{jones2}). Moreover, it is crucial to consider the geometric part, and not the whole image of the representation. In fact, for instance one can show that $\phi=(x+t)^2+1\in \Z[t][x]$ has surjective arboreal representation, but $\rho_{\phi_p}$ is certainly not surjective for every $p\equiv 1 \bmod 4$.

Finally, we will show how to use Theorem \ref{gal_rep} to prove the following theorem, which resolves \cite[Conjecture 6.7]{jones1}.

\begin{theorem}\label{jones_conjecture}
 Let $k$ be a field of characteristic $\neq 2$, and let $t$ be transcendental over $k$. Then the polynomial $\phi=x^2+t\in k(t)[x]$ has surjective arboreal representation.
\end{theorem}
As explained in \cite[Section 6]{jones1}, Theorem \ref{jones_conjecture} has deep consequences in the study of the arithmetic of the map $x^2+c$ with $c\in \overline{\vF}_p$, which is related to the $p$-adic Mandelbrot set.

Here is a brief outline of the paper: in Section \ref{sec:bounds} we will prove Theorem \ref{main_thm}, by providing (cf.\ Theorems \ref{height_bound} and \ref{height_bound_sing}) a suitable completely effective height bound for certain integral points on elliptic curves over function fields in positive characteristic associated to quadratic polynomials. In Section \ref{sec:reduction} we will prove Theorem \ref{gal_rep}, using Theorem \ref{main_thm} and the relation between the squarefree Zsigmondy set and the image of an arboreal Galois representation (cf.\ Corollary \ref{finite_index_cor}). Finally, in Section \ref{sec:conjecture} we will prove Theorem \ref{jones_conjecture}.

\subsection*{Notation and conventions}
When $K/\overline{\F}_p(t)$ is a finite extension, we will denote by $M_K$ a complete set of valuations of $K$. All valuations in $M_K$ are normalized, and all residue fields have degree 1, since the base field is algebraically closed. We will denote by $M_K^0$ the set of finite valuations, i.e.\ the set of valuations of $K$ not extending the infinite valuation of $\overline{\F}_p(t)$, and by $M_K^{\infty}$ the set of infinite valuations, so that $M_K=M_K^0\sqcup M_K^{\infty}$.

For an element $f\in K$, the logarithmic height of $f$, with respect to $K$, is defined by:
$$h_K(f)=\sum_{v\in M_K}\max\{v(f),0\}.$$
When $K=\overline{\F}_p(t)$, we will omit $K$ from the notation for height, i.e.\ we will write $h(f)$ for $h_{\overline{\F}_p(t)}(f)$.

If $f\in \overline{\F}_p[t]$, the degree of $f$ coincides with $h(f)$. We denote by $\rad f$ the radical of $f$, namely the product of its distinct irreducible factors.

If $L/K$ is a finite, separable field extension, we will denote by $N_{L/K}\colon L^{\times}\to K^{\times}$ the norm map.

If $G$ is a group acting on a set $S$, the action of $\sigma\in G$ on $x\in S$ will be denote in the upper left corner, as in ${}^{\sigma}x$.

In the rest of the paper whenever we take a square root of an element $r$ we are implicitly chosing one of the roots $r_1$ and $-r_1$ of $x^2-r$. Whenever we do that, such choice is fixed for the rest of the paper.

\section{Bounding \texorpdfstring{$\mathcal Z_s(\phi_\p)$}{}}\label{sec:bounds}
The goal of this section is to prove Theorem \ref{main_thm}. We will start with some auxiliary lemmas. Throughout the whole section, $p$ will be an odd prime and $\overline{\F}_p$ will be a fixed algebraic closure of $\F_p$.

\begin{lemma}[ABC in positive characteristic]\label{abc}
 Let $K/\overline{\F}_p(t)$ be a finite, separable extension, and let $\gamma_1,\gamma_2,\gamma_3\in K$ be non-zero elements such that $\gamma_1+\gamma_2+\gamma_3=0$. Let $V\subset M_K$ be a finite set such that $v(\gamma_1)= v(\gamma_2)=v(\gamma_3)$ for all $v\notin V$ and let $g$ be the genus of $K$. Suppose $\gamma_2/\gamma_3\in K^{p^e}\setminus K^{p^{e+1}}$ for some $e\geq 0$. Then:
 $$h_K(\gamma_2/\gamma_3)\leq p^e(2g-2+|V|).$$
\end{lemma}

\begin{proof}
Let $\gamma_2/\gamma_3=f^{p^e}$, for some $f\in K\setminus K^p$. Then $\gamma_1/\gamma_3=(-f-1)^{p^e}$. Now apply \cite[Lemma 10]{mason1} to the triple $(-f-1,f,1)$: this yields $h_K(f)\leq 2g-2+|V|$. It is immediate to conclude by noticing that $h_K(\gamma_2/\gamma_3)= p^e h_K(f)$.
\end{proof}

From now on, we let $\gamma,c_1\in\overline{\F}_p[t]$ and $\phi=(x-\gamma)^2-c_1\in \overline{\F}_p[t][x]$. We assume throughout the section that $c_1\neq 0$ and $\deg(\gamma+c_1)>0$, so that $\phi$ is non-square and non-isotrivial. Let $\{c_n\}_{n\geq 1}$ be the adjusted post-critical orbit of $\phi$. Following \cite{hindes1}, we let $h(\phi)\coloneqq\max\{h(\gamma),h(c_1)\}$.

\begin{lemma}[{{\cite[Lemma 2]{hindes1}}}]\label{height_lemma}
The following height bounds hold.
\begin{enumerate}
 \item Suppose that $h(\gamma)\neq h(c_1)$. Then:
 \begin{enumerate}[a)]
 \item $h(c_n)=2^{n-1}\cdot h(\phi)$ for all $n\geq 2$, and $h(c_1)\leq h(\phi)$;
 \item $h(\phi^{(n)}(0))\leq 2^n\cdot h(\phi)$ for all $n\geq 1$.
 \end{enumerate}
 \item Suppose that $h(\gamma)=h(c_1)$ and let $\kappa_\phi\coloneqq \log_2\left(\frac{h(\gamma)}{h(\gamma+c_1)}\right)+1$. Then:
 \begin{enumerate}[a)]
 \item $h(c_n)\leq h(\gamma)$ for all $n\leq \kappa_\phi$;
 \item $h(c_n)=2^{n-1}\cdot h(\gamma+c_1)$ for all $n>\kappa_\phi$;
 \item $h(\phi^{(n)}(0))=2^nh(\gamma)$ for all $n\geq 1$.
 \end{enumerate}
\end{enumerate}
 \end{lemma}
The following lemma shows that for a large enough $n$ and a fixed $D\in \overline{\F}_p[t]$, the polynomial $c_n-D$ cannot be a square.
\begin{lemma}\label{lemma:cnnonsquare}
Let $D\in\overline{\F}_p[t]$ be such that $c_1+D\neq 0$ and let $n\in\N$ be such that $h(\gamma+c_1)\cdot 2^{n-2}>h(c_1+D)$. Then  $c_n-D$ is not a square in $\overline{\F}_p(t)$.
\end{lemma}
\begin{proof}

Suppose by contradiction that $c_n-D$ is a square. Then for some $g\in \overline{\vF}_p[t]$ we have
\[(c_{n-1}-\gamma)^2-c_1-D=g^2,\]
yielding:
\[(c_{n-1}-\gamma-g)(c_{n-1}-\gamma+g)=c_1+D.\]
If we can show that $h(c_{n-1}-\gamma)>h(c_1+D)$ we are done, as the leading coefficient of $c_{n-1}-\gamma$ cannot be deleted by the leading coefficient of $g$ both with $+$ and $-$ (and also none of the two factors can obviously be zero, as $c_1+D\neq 0$). 
Observe now that $c_{n-1}-\gamma$ can be seen as the evaluation of the polynomial
\[F\coloneqq (\cdots(Y^2+Y)^2+Y)^2+Y)...)^2+Y\in \F_p[Y]\]
 at $-c_1-\gamma$.
So $h(c_{n-1}-\gamma)=h(F(-c_1-\gamma))=h(\gamma+c_1)2^{n-2}$. But this is greater than $h(c_1+D)$ by assumption, so we are done.
\end{proof}

In order to state and prove the key result, we need to introduce the following definition. 

\begin{definition}\label{def:insepdeg}
 The \emph{dynamical inseparability degree} of $\phi$ is the non-negative integer $e$ defined in the following way:
$$e\coloneqq\begin{cases}
     \max\{i\in\N\colon c_2/c_1\in \overline{\F}_p(t)^{p^i}\} & \mbox{ if }c_1\notin \overline{\F}_p[t]^2\\
     \max\{i\in\N\colon \frac{\gamma}{\sqrt{c_1}}+\sqrt{c_1}\in \overline{\F}_p(t)^{p^i}\} & \mbox{ if } c_1\in\overline{\F}_p[t]^2 \mbox{ and } \frac{\gamma}{\sqrt{c_1}}+\sqrt{c_1}\notin \overline{\F}_p\\
     \max\{i\in\N\colon c_1\in\overline{\F}_p(t)^{p^i}\} & \mbox{ if }c_1\in\overline{\F}_p[t]^2 \mbox{ and } \frac{\gamma}{\sqrt{c_1}}+\sqrt{c_1}\in\overline{\F}_p\\
    \end{cases}.
$$
\end{definition}

Notice that the dynamical inseparability degree is well-defined: if $c_1$ is not a square, then $c_2/c_1$ cannot be constant. If $c_1$ is a square and $\frac{\gamma}{\sqrt{c_1}}+\sqrt{c_1}$ is constant, then $c_1$ cannot be constant because otherwise $\phi$ would be isotrivial. Notice also that the definition does not depend on the choice of a square root of $c_1$.

Definition \ref{def:insepdeg} and Lemmas \ref{separability} and \ref{separability_sing} constitute the technical heart of the argument in this paper. In fact, they allow to transfer a technique to obtain height bounds for integral points on hyperelliptic curves in characteristic zero due to Baker \cite{baker1969bounds} and Mason \cite{mason1983hyperelliptic} to heights bounds for elements in the adjusted post-critical orbit of $\phi$ in positive characteristic. This will be described in Theorems \ref{height_bound} and \ref{height_bound_sing}.

Now we need to distinguish two cases, according to whether $\gamma+c_1\pm\sqrt{c_1}=0$ or not. Although the arguments in the two cases are essentially the same, certain crucial elements constructed starting from $\phi$ are different. Hence, in order not to confuse the reader, we will split the two cases in two different subsections.

\subsection{The case \texorpdfstring{$\gamma+c_1\pm\sqrt{c_1}\neq 0$}{}.}Let $\phi=(x-\gamma)^2-c_1\in\overline{\F}_p[t]$ be non-isotrivial, with $c_1\neq 0$ and such that $\gamma+c_1\pm\sqrt{c_1}\neq 0$. Let us fix $n\geq 2$ and define the following quantities, which we will use in the rest of the subsection:
\begin{equation}\label{eq:alphaxi}
\alpha_1\coloneqq \gamma-\sqrt{c_1}, \,\, \alpha_2\coloneqq \gamma+\sqrt{c_1}, \,\, \alpha_3\coloneqq -c_1, \,\, \xi_i\coloneqq \sqrt{c_{n-1}-\alpha_i} \text{ for } i\in\{1,2,3\},
\end{equation}

\begin{equation}\label{eq:beta}
\beta_1=\xi_2-\xi_3, \quad \beta_2=\xi_3-\xi_1, \quad\beta_3=\xi_1-\xi_2,
\end{equation}
and
\begin{equation}\label{eq:betahat}
\widehat{\beta}_1=\xi_2+\xi_3, \quad \widehat{\beta}_2=\xi_3+\xi_1,\quad \widehat{\beta}_3=\xi_1+\xi_2.
\end{equation}
 Notice that $\alpha_i\neq \alpha_j$ whenever $i\neq j$ and that consequently $\beta_i,\widehat{\beta}_i\neq 0$ for every $i\in\{1,2,3\}$. Notice also that by construction, $\xi_3=\sqrt{c_{n-1}+c_1}=c_{n-2}-\gamma\in \overline{\F}_p[t]$. Finally, let $K\coloneqq \overline{\F}_p(t,\alpha_1,\alpha_2,\alpha_3)=\overline{\F}_p(t,\alpha_1)$ and $L_n\coloneqq K(\xi_1,\xi_2,\xi_3)=K(\xi_1,\xi_2)$. 

\begin{lemma}\label{separability}
 Let $n\geq 3$ be such that $h(\gamma+c_1)2^{n-3}>h(\phi)$. Let $\displaystyle y\in\left\{\frac{\beta_2}{\beta_3},\frac{\widehat{\beta}_2}{\beta_3},\frac{\beta_2}{\widehat{\beta}_3},\frac{\widehat{\beta}_2}{\widehat{\beta}_3}\right\}$, let $e$ be the dynamical inseparability degree of $\phi$ and assume that $y\in L_n^{p^i}$ for some $i\in \N$. Then $i\leq e$.
\end{lemma}
\begin{proof}
 First of all, notice that:
 $$\beta_1\widehat{\beta}_1=\alpha_3-\alpha_2,\quad \beta_2\widehat{\beta}_2=\alpha_1-\alpha_3,\quad \beta_3\widehat{\beta}_3=\alpha_2-\alpha_1.$$
 This immediately shows that, since $[L_n:K]=2^k$ for some $k\in \vN$, then:
 \begin{equation}\label{norms}
  N_{L_n/\overline{\F}_p(t)}\left(\frac{\beta_2\widehat{\beta}_2}{\beta_3\widehat{\beta}_3}\right)=\begin{cases}
                                                                                                  \frac{1}{4^{2^k}}\left(\frac{c_2}{c_1}\right)^{2^k} & \mbox{ if } c_1\notin \overline{\F}_p[t]^2\\
                                                                                                  \frac{1}{2^{2^k}}\left(\frac{\gamma}{\sqrt{c_1}}+\sqrt{c_1}-1\right)^{2^k} & \mbox{ if } c_1\in\overline{\F}_p[t]^2\\
                                                                                                 \end{cases}.
  \end{equation}
  Next, we claim that:
 \begin{equation}\label{eq1}
  N_{L_n/\overline{\F}_p(t)}\left(\frac{\beta_2}{\beta_3}\right)=N_{L_n/\overline{\F}_p(t)}\left(\frac{\widehat{\beta}_2}{\beta_3}\right)=N_{L_n/\overline{\F}_p(t)}\left(\frac{\beta_2}{\widehat{\beta}_3}\right)=N_{L_n/\overline{\F}_p(t)}\left(\frac{\widehat{\beta}_2}{\widehat{\beta}_3}\right). \end{equation}
 Relations \eqref{norms} and \eqref{eq1} immediately prove that if $c_1\notin \overline{\F}_p[t]^2$ or $c_1\in \overline{\F}_p[t]^2$ but $\frac{\gamma}{\sqrt{c_1}}+\sqrt{c_1}$ is not constant, then the statement of the lemma holds true.
 
 We will prove \eqref{eq1} by examining separately the cases $c_1\in \overline{\F}_p[t]^2$ (i.e.\ $K=\overline{\F}_p(t)$) and $c_1\notin \overline{\F}_p[t]^2$ (i.e.\ $[K:\overline{\F}_p(t)]=2$).

 \textbf{Case 1): $c_1\notin \overline{\F}_p[t]^2$}. We prove first that $\xi_1,\xi_2\notin K=\overline{\F}_p(t,\sqrt{c_1})$. In fact, suppose by contradiction that $\xi_i\in K$, where $i\in \{1,2\}$. Then $\xi_1^2=v^2$ and $\xi_2^2={{}^{\sigma}v}^2$ for some $v\in K$, where $\sigma$ is the generator of $\Gal(K/\overline{\F}_p(t))$. It follows that $(v{}^{\sigma}v)^2=(c_{n-1}-\gamma)^2-c_1=c_n$, which implies in particular that $c_n\in\overline{\F}_p[t]^2$. By Lemma \ref{lemma:cnnonsquare}, this cannot happen. We now have to distinguish two subcases.
 
 \textbf{Subcase 1a): $\xi_1\notin K(\xi_2)$}. In this case, $L_n/K$ is a Galois extension with Galois group $C_2\times C_2$. It follows immediately that there exists $\tau$ in $\Gal(L_n/K)$ such that ${}^{\tau}\xi_1=-\xi_1$ and ${}^{\tau}\xi_2=\xi_2$. Therefore, ${}^{\tau}\beta_2=\widehat{\beta}_2$ and ${}^{\tau}\beta_3=-\widehat{\beta}_3$, proving \eqref{eq1}.
 
 \textbf{Subcase 1b): $\xi_1\in K(\xi_2)$}. Suppose $(\xi_1\cdot\xi_2)^2=u^2$ for some $u\in K$. Writing $u=A+B\sqrt{c_1}$ for some $A,B\in\overline{\F}_p(t)$, one sees immediately, again by Lemma \ref{lemma:cnnonsquare}, that $A=0$ must hold, so that $(\xi_1\cdot\xi_2)^2=c_n=B^2\cdot c_1$. It is then a well-known Galois theoretic fact that $L_n/\overline{\F}_p(t)$ is a Galois extension with cyclic Galois group of order 4. Since $\xi_1$ generates $L_n$ over $\overline{\F}_p(t)$, then there exists a generator $\nu$ of such group with the property that ${}^{\nu}\xi_1=\xi_2$ while ${}^{\nu}\xi_2=-\xi_1$. It follows immediately that ${}^{\nu^2}\xi_1=-\xi_1$, and hence ${}^{\nu}\beta_3=\widehat{\beta}_3$ and ${}^{\nu^2}\beta_2=\widehat{\beta}_2$, proving \eqref{eq1} again.
 
 \textbf{Case 2): $c_1\in \overline{\F}_p[t]^2$}. Notice that if $\xi_i\in K$ for some $i\in\{1,2\}$, then $c_{n-1}-\alpha_i\in K^2$, which is impossible by Lemma \ref{lemma:cnnonsquare}. Moreover, if $\xi_1\xi_2\in K$ then $c_n=\xi_1^2\xi_2^2\in K^2$, which is impossible by Lemma \ref{lemma:cnnonsquare} again. Hence, $L_n/K$ is Galois with Galois group $C_2\times C_2$. It is immediate to see that $\beta_2/\beta_3,\widehat{\beta}_2/\beta_3,\beta_2/\widehat{\beta}_3,\widehat{\beta}_2/\widehat{\beta}_3$ are all Galois conjugates up to sign, and \eqref{eq1} follows.
 \vspace{3mm}
 
 It only remains to prove the lemma in the case where $c_1\in \overline{\F}_p[t]^2$ and $\frac{\gamma}{\sqrt{c_1}}+\sqrt{c_1}$ is a non-zero constant. This forces $\gamma=-c_1-u\sqrt{c_1}$ for some $u\in \overline{\F}_p^{\times}$. Since, as we showed in Case 2), $\beta_2/\beta_3, \widehat{\beta}_2/\beta_3, \beta_2/\widehat{\beta}_3, \widehat{\beta}_2/\widehat{\beta}_3$ are all Galois conjugate up to sign, it is enough to assume $y=\beta_2/\beta_3$. By hypothesis, $y\in L_n^{p^i}$ for some $i$. Write $\beta_2/\beta_3=(A+B\xi_1+C\xi_2+D\xi_1\xi_2)^{p^i}$ for some $A,B,C,D\in \overline{\F}_p(t)$. Then:
 $$\frac{\beta_2}{\beta_3}=\frac{\beta_2\widehat{\beta}_3}{\beta_3\widehat{\beta}_3}=\frac{-\xi_1^2+(c_{n-2}-\gamma)\xi_1+(c_{n-2}-\gamma)\xi_2-\xi_1\xi_2}{2\sqrt{c_1}}.$$
Simple algebraic manipulations, together with the fact that $\{1,\xi_1,\xi_2,\xi_1\xi_2\}$ is a $K$-basis of $L_n$, show that one must have $(D^2c_n)^{p^i}=c_n/(4c_1)$, implying that $c_n/c_1\in \overline{\F}_p(t)^{p^i}$. An easy induction shows that for $n\geq 3$, $c_n/c_1=\sqrt{c_1}\cdot (\sqrt{c_1}\cdot g_n+2u^3)+u^2-1$ for some $g_n\in\overline{\F}_p[t]$. Now observe that $c_n/c_1$ is a $p^i$-th power if and only if $c_n/c_1-(u^2-1)$ is, and that $\sqrt{c_1}$ is coprime with $\sqrt{c_1}\cdot g_n+2u^3$. This forces $c_1$ to lie in $\overline{\F}_p(t)^{p^i}$, and therefore $i\leq e$.
\end{proof}

Motivated by an idea of Hindes, we will use height bounds to study the Zsigmondy set of $\phi$ (see for example \cite{hindes1,hindjon}). From now on, for any $n\geq 1$ we write $c_n=d_ny_n^2$, where $y_n,d_n\in \overline{\F}_p[t]$ and $d_n$ is squarefree. Let us define the following elliptic curve over $\overline{\F}_p(t)$:
 $$E_\phi\colon y^2=(x+c_1)((x-\gamma)^2-c_1)\quad \left(=\prod_{i=1}^3(x-\alpha_i)\right).$$
 The point $(X_n,Y_n)\coloneqq (c_{n-1},\sqrt{d_n}y_n(c_{n-2}-\gamma))$ lies on $E_\phi$.

 \begin{theorem}\label{height_bound}
 Let $\phi=(x-\gamma)^2-c_1\in \overline\vF_p(t)[x]$ be non-isotrivial, non-square and having dynamical inseparability degree $e$. Suppose that $c_1+\gamma\pm\sqrt{c_1}\neq 0$ and let $n\geq 3$ be such that $\deg(\gamma+c_1)2^{n-3}>h(\phi)$. Then we have:
  $$h(c_{n-1})\leq A\cdot h(d_n)+B,$$
  where
  $$A=8p^e \quad\mbox{and}\quad B=8p^e(h(d_1)+4+4h(\rad c_1)+4h(\rad c_2)) +4h(\gamma)+8h( c_1).$$
 \end{theorem}

\begin{proof}
 Let $\alpha_i,\xi_i,\beta_i,\widehat \beta_i$ be as in equations \eqref{eq:alphaxi}, \eqref{eq:beta}, \eqref{eq:betahat}. Moreover, recall that 
$K=\overline{\F}_p(t,\alpha_1)$ and $L_n=K(\xi_1,\xi_2)$.
Let $K_n\coloneqq K(\sqrt{d_n})$. Notice that $\sqrt{d_n}\in L_n$, and therefore $K_n\subseteq L_n$. Set $2^\delta=[L_n:K_n]$. Define \[\overline X_n\coloneqq\frac{2(2c_{n-1}-\alpha_1-\alpha_3)}{\alpha_2-\alpha_1}.\]

The idea of the proof is to start by observing the following identity:

\[\overline X_n=\frac{\widehat{\beta}_2^2}{\beta_3\widehat{\beta}_3}+\frac{\beta_2^2}{\beta_3\widehat{\beta}_3}.\]

Now the fact that if $a,b\in L_n$ then $\max\{h_{L_n}(ab),h_{L_n}(a+b)\}\leq h_{L_n}(a)+h_{L_n}(b)$ immediately shows that:
$$h_{L_n}(\overline X_n)\leq h_{L_n}(\widehat{\beta}_2/\beta_3)+h_{L_n}
(\widehat{\beta}_2/\widehat\beta_3)+h_{L_n}(\beta_2/\beta_3)+h_{L_n}(\beta_2/\widehat\beta_3).$$

Since we have, by construction, that: $$\beta_1+\beta_2+\beta_3=\beta_1+\widehat{\beta}_2-\widehat{\beta}_3=-\widehat{\beta}_1+\beta_2+\widehat{\beta}_3=\widehat{\beta}_1-\widehat{\beta}_2+\beta_3=0,$$
setting $V=M_{L_n}\setminus \{v\in M_{L_n}\colon v(\beta_i)=v(\widehat \beta_i)=0\;\forall i\in\{1,2,3\}\}$ we obtain, by Lemmas \ref{abc} and \ref{separability}, that
\begin{equation}\label{fundamental_inequality}
2^{\delta}h_{K_n}(\overline{X}_n)=h_{L_n}(\overline X_n)\leq 4p^e(2g_{L_n}-2+|V|),
\end{equation}
where $g_{L_n}$ is the genus of $L_n$. Now the rest of the proof will consist of obtaining suitable bounds on the terms of inequality \eqref{fundamental_inequality}.

We start by finding an upper bound on the right hand side of \eqref{fundamental_inequality}. In order to do so, first let $\mu\coloneqq (\alpha_3-\alpha_2)(\alpha_1-\alpha_3)(\alpha_2-\alpha_1)\in K_n\setminus\{0\}$. We claim that:
\begin{equation}\label{mu_integral}
 V\subseteq M_{L_n}^{\infty}\cup \{v\in M_{L_n}^0\colon v(\mu)>0\}.
\end{equation}

In fact, since the $\beta_i$'s and the $\widehat{\beta}_j$'s lie in the integral closure of $\overline{\F}_p[t]$ inside $L_n$, then for every $v\in M_{L_n}^0$ we have $v(\beta_i),v(\widehat{\beta}_j)\geq 0$ for all $i,j\in \{1,2,3\}$. Thus if $v\in V$ is finite, we have $v(\beta_i)>0$ or $v(\widehat \beta_i)>0$ for some $i\in \{1,2,3\}$. Now observe that $\mu=\beta_1\widehat \beta_1\beta_2\widehat \beta_2\beta_3\widehat \beta_3$; it follows that $v(\mu)>0$.

Therefore, if
$$W\coloneqq M_{K_n}^{\infty}\cup \{w\in M_{K_n}^0\colon w(\mu)>0\},$$
relation \eqref{mu_integral} immediately shows that:
\begin{equation}\label{bound_V}
|V|\leq 2^{\delta}|W|.
\end{equation}

Next, we are going to deal with the term $2g_{L_n}-2$ of \eqref{fundamental_inequality}. Since $L_n/K_n$ is a Galois extension, Hurwitz formula shows that:
\begin{equation}\label{hurwitz}
2g_{L_n}-2=2^{\delta}(2g_{K_n}-2)+2^{\delta}\sum_{v\in M_{K_n}}\frac{e_v-1}{e_v}\leq 2^{\delta}(g_{K_n}-2)+2^{\delta}|R|,
\end{equation}
where $g_{K_n}$ is the genus of $K_n$, $e_v$ is the ramification index at $v$, and $R$ is the set of valuations of $K_n$ that ramify in $L_n$. We now claim that $R\subseteq W$. To show that, let $v\in M_{K_n}^0$ be such that $v(\mu)=0$. We will show that $v\notin R$. Notice that if $v(\mu)=0$ then there cannot be $i,j\in\{1,2,3\}$ with $i\neq j$ such that $v(X_n-\alpha_i),v(X_n-\alpha_j)>0$, as otherwise we would have $v(\alpha_i-\alpha_j)\geq \min\{v(X_n-\alpha_i),v(X_n-\alpha_j)\}>0$, yielding a contradiction (notice that $v(\alpha_i-\alpha_j)\geq 0$ for any $v\in M_{K_n}^0$, because $\alpha_i,\alpha_j$ belong to the integral closure of $\overline{\F}_p[t]$ inside $K_n$). Therefore, either $v(X_n-\alpha_i)=0$ for all $i\in\{1,2,3\}$, or $v(X_n-\alpha_i)=2v(Y_n)$ for exactly one $i\in \{1,2,3\}$, as $Y_n^2=\prod_{i=1}^3(X_n-\alpha_i)$. Now simply recall that $L_n=K_n(\sqrt{X_n-\alpha_1},\sqrt{X_n-\alpha_2})$: if $v(X_n-\alpha_i)$ is even for all $i\in\{1,2\}$, then $v$ does not ramify in $L_n$ (see for example \cite[Proposition 3.7.3]{stich}).

The fact that $R\subseteq W$, together with \eqref{fundamental_inequality}, \eqref{bound_V} and \eqref{hurwitz}, shows that:
\begin{equation}\label{fundamental_inequality_2}
 h_{K_n}(\overline X_n)\leq 4p^e(2g_{K_n}-2+2|W|).
\end{equation}
The genus of $K_n$ can be bounded in the following way: since $K_n$ is the compositum of $K$ and $\overline{\F}_p(t,\sqrt{d_n})$, one can use Castelnuovo's Inequality (see \cite[Theorem 3.11.3]{stich}) to get
$$2g_{K_n}-2\leq 4g_n+4g_K,$$
where $g_n$ is the genus of $\overline{\F}_p(t,\sqrt{d_n})$ and $g_K$ is the genus of $K$. On the other hand, \cite[Proposition 6.2.3]{stich} shows that $g_n\leq \frac{h(d_n)}{2}$, and \eqref{fundamental_inequality_2} becomes
   $$h_{K_n}(\overline X_n)\leq 8p^e(h(d_n)+2g_K+|W|).$$
To deal with the left hand side of the above inequality, just notice that by standard properties of the height we have:
$$h_{K_n}(\overline X_n)\geq h_{K_n}(c_{n-1})-h_{K_n}(\alpha_1+\alpha_3)-h_{K_n}(\alpha_2-\alpha_1),$$
yielding:
\begin{equation}\label{final_inequality}
 h_{K_n}(c_{n-1})\leq 8p^e(h(d_n)+2g_K+|W|)+h_{K_n}(\alpha_1+\alpha_3)+h_{K_n}(\alpha_2-\alpha_1).
\end{equation}
 To conclude the proof, it is now enough to find explicit bounds on the terms of the above inequality.
 
 The genus of $K$ is simply bounded by $\frac{h(d_1)}{2}$, where $d_1$ is the squarefree part of $c_1$ (see \cite[Proposition 6.2.3]{stich}).
 
 Next, we have $\mu=-2\sqrt{c_1}\cdot c_2$, and therefore if $w\in W$ is a finite valuation then the valuation $v$ of $\overline{\F}_p(t)$ lying below $w$ is such that $v(c_1c_2)>0$. Clearly there are at most $h(\rad c_1)+h(\rad c_2)$ finite valuations $v$ of $\overline{\F}_p(t)$ such that $v(c_1c_2)>0$, and therefore $|W|\leq |M_{K_n}^{\infty}|+ 4h(\rad c_1)+4h(\rad c_2)\leq 4+4h(\rad c_1)+4h(\rad c_2)$.
 
 Finally, notice that $\alpha_2-\alpha_1=2\sqrt{c_1}$, and hence $h_K(\sqrt{c_1})\leq h(c_1)$, yielding $h_{K_n}(\alpha_2-\alpha_1)\leq 2h(c_1)$. On the other hand $\alpha_1+\alpha_3=\gamma-c_1-\sqrt{c_1}$, so that $h_{K_n}(\alpha_1+\alpha_3)\leq 4h(\gamma)+6h( c_1)$.
\end{proof}

\subsection{The case \texorpdfstring{$\gamma+c_1\pm\sqrt{c_1}=0$}{}.}

 When $c_1\in\overline{\F}_p[t]^2$ is not constant and $\phi=(x+c_1\pm\sqrt{c_1})^2-c_1$, the curve $E_\phi$ we used for proving Theorem \ref{height_bound} is singular and the arguments used in the proof fail, because $\mu=(\alpha_3-\alpha_2)(\alpha_1-\alpha_3)(\alpha_2-\alpha_1)=0$. The goal of this subsection is to explain how to modify the curve and the arguments in order to obtain, also in this case, a bound analogous to the one of Theorem \ref{height_bound}.
 
 From now on, fix $\eta\in \{\pm1\}$ and let $\gamma\coloneqq -c_1-\eta\sqrt{c_1}$ and $\phi=(x-\gamma)^2-c_1$. Let $\{c_n\}$ be the adjusted post-critical orbit of $\phi$.
 
 First, we need a version of Lemma \ref{separability}. This time we let
 $$\alpha_1\coloneqq -c_1-\eta\sqrt{c_1}-\sqrt{-2\sqrt{c_1}}, \quad \alpha_2\coloneqq -c_1-\eta\sqrt{c_1}+\sqrt{-2\sqrt{c_1}}, \quad \alpha_3\coloneqq -c_1,$$
 and $\xi_i\coloneqq\sqrt{c_{n-2}-\alpha_i}$, for $i\in\{1,2,3\}$. Notice that $\xi_3\in\overline{\F}_p[t]$ once again. The fields $K$ and $L_n$ and the $\beta_j$'s are defined in the same way as above Lemma \ref{separability}, so in particular we will have $K=\overline{\F}_p(t,\alpha_1,\alpha_2,\alpha_3)=\overline{\F}_p(t,\sqrt[4]{c_1})$ and $L_n=K(\xi_1,\xi_2)$. Notice that $\alpha_i\neq\alpha_j$ whenever $i\neq j$. According to Definition \ref{def:insepdeg}, the dynamical inseparability degree of $\phi$ in this case is given by:
 $$e\coloneqq \max\{i\in \N\colon c_1\in \overline{\F}_p(t)^{p^i}\}.$$

\begin{lemma}\label{separability_sing}
 Let $n\geq 4$ and let $\displaystyle y\in\left\{\frac{\beta_2}{\beta_3},\frac{\widehat{\beta}_2}{\beta_3},\frac{\beta_2}{\widehat{\beta}_3},\frac{\widehat{\beta}_2}{\widehat{\beta}_3}\right\}$. If $y\in L_n^{p^i}$, then $i\leq e$.
\end{lemma}
\begin{proof}
One starts by computing:

$$f\coloneqq N_{L_n/\overline{\F}_p(t)}\left(\frac{\beta_2\widehat{\beta}_2}{\beta_3\widehat{\beta}_3}\right)=\begin{cases}
                                                                                                  \left(-\frac{1}{8}\sqrt{c_1}-\frac{1}{4}\right)^{2^k} & \mbox{if }\sqrt{c_1}\notin \overline{\F}_p[t]^2\\
                                                                                                  \left(-\frac{1}{2\sqrt{2}}\sqrt[4]{c_1}-\frac{1}{2}\right)^{2^k} & \mbox{if }\sqrt{c_1}\in\overline{\F}_p[t]^2\\
                                                                                                 \end{cases},
$$
where $2^k=[L_n:K]$. Notice that $f$ is not constant, and if $f\in \overline{\F}_p(t)^{p^i}$ then clearly $i\leq e$. Now the proof is essentially the same as that of Lemma \ref{separability}, even slightly easier since $f$ cannot be constant. Therefore, in order to conclude it is sufficient to show that:
$$N_{L_n/ \overline{\F}_p(t)}\left(\frac{\beta_2}{\beta_3}\right)=N_{L_n/\overline{\F}_p(t)}\left(\frac{\widehat{\beta}_2}{\beta_3}\right)=N_{L_n/\overline{\F}_p(t)}\left(\frac{\beta_2}{\widehat{\beta}_3}\right)=N_{L_n/\overline{\F}_p(t)}\left(\frac{\widehat{\beta}_2}{\widehat{\beta}_3}\right),$$
 and this works precisely as in Lemma \ref{separability}: one shows that $L_n/K$ is either a $C_2\times C_2$-extension or a $C_2$-extension which is Galois over $\overline{\F}_p(t)$ with Galois group $C_4$. This implies once again that the elements we are looking at are all Galois conjugate up to sign. Notice that given the particular shape of $\gamma$ and the fact that $c_1$ is not constant, the hypothesis $n\geq 4$ is enough to conclude. 
\end{proof}
 
 Next, observe that $\phi=\phi_1\phi_2$, where $\phi_1=x+c_1+2\eta\sqrt{c_1}$ and $\phi_2=x+c_1$. Notice that $\phi_2\circ \phi^{(n)}$ is a square polynomial for every $n$. For every $n\geq 2$ write $\phi_1(c_{n-1})=d_ny_n^2$, where $d_n,y_n\in \overline{\F}_p[t]$ and $d_n$ is squarefree. It follows that: 
 $$c_n=\phi(c_{n-1})=\phi_1(c_{n-1})\phi_2(c_{n-1})=d_nz_n^2 \mbox{ for some } z_n\in \overline{\F}_p[t],$$
 so that $d_n$ coincides with the squarefree part of $c_n$.
 
 Now define $E_{\phi}\colon y^2=\prod_{i=1}^3(x-\alpha_i)$, which is again an elliptic curve.
 The point
 $$(X_n,Y_n)\coloneqq (c_{n-2},\sqrt{d_n}y_n(c_{n-3}-\gamma))$$
 lies on $E_{\phi}$.

 The analogue of Theorem \ref{height_bound} is the following.
 \begin{theorem}\label{height_bound_sing}
  Let $\phi=(x+c_1+\eta\sqrt{c_1})^2-c_1$ be non-isotrivial, having dynamical inseparability degree $e$. Suppose that $n\geq 4$ and let $\widetilde{d}_1$ be the squarefree part of $\sqrt{c_1}$. Then we have:
  $$h(c_{n-2})\leq A\cdot h(d_n)+B,$$
  where
  $$A=8p^e \mbox{ and } B=8p^e(h(\widetilde{d}_1)+4+4h(\rad(c_1))+4h(\rad(\sqrt{c_1}+2)))+10h(c_1).$$
 \end{theorem}
\begin{proof}
 The proof follows verbatim the one of Theorem \ref{height_bound} up until equation \eqref{final_inequality}, with the sole difference that $c_{n-1}$ is replaced by $c_{n-2}$. Of course, in the course of the proof we have to use Lemma \ref{separability_sing} instead of Lemma \ref{separability}, obtaining nevertheless the very same conclusion. Recall that inequality \eqref{final_inequality} now reads:
 $$h_{K_n}(c_{n-2})\leq 8p^e(h(d_n)+2g_K+|W|)+h_{K_n}(\alpha_1+\alpha_3)+h_{K_n}(\alpha_2-\alpha_1),$$
 where $g_K$ is the genus of $K$ and $W$ is the set of valuations $w$ of $K_n$ such that $w$ is infinite or $w(\mu)>0$, where $\mu=(\alpha_3-\alpha_2)(\alpha_1-\alpha_3)(\alpha_2-\alpha_1)$.
 
 The genus of $K$ can be bounded again by \cite[Proposition 6.2.3]{stich}, yielding $g_K\leq \frac{h(\widetilde{d}_1)}{2}$.
 
 Next, in this case $\mu=-2\sqrt{-2}\sqrt[4]{c_1}(c_1+2\sqrt{c_1})$. Thus, if $w\in W$ is a finite valuation with $w(\mu)>0$, the valuation $v$ of $\overline{\F}_p(t)$ lying below $\mu$ is such that $v(\sqrt{c_1}(c_1+2\sqrt{c_1}))>0$. Clearly there are at most $\deg\rad(c_1)+\deg\rad(\sqrt{c_1}+2)$ such finite valuations of $\overline{\F}_p(t)$, and it follows that $|W|\leq 4+4\deg\rad(c_1)+4\deg\rad(\sqrt{c_1}+2)$.
 
 Finally, $\alpha_2-\alpha_1=2\sqrt{-2\sqrt{c_1}}$ and $\alpha_1+\alpha_3=-2c_1-\eta\sqrt{c_1}-\sqrt{-2\sqrt{c_1}}$. Notice that $h_K(\sqrt[4]{c_1})\leq h(c_1)$, so that $h_{K_n}(\alpha_2-\alpha_1)\leq 2h(c_1)$ and $h_{K_n}(\alpha_1+\alpha_3)\leq 8h(c_1)$.
\end{proof}

\subsection{The proof of Theorem \ref{main_thm}.}

Let us start by explaining how Theorems \ref{height_bound} and \ref{height_bound_sing} yield a bound on the squarefree Zsigmondy set, adapting a beautiful trick due to Hindes \cite{hindes1}.

\begin{corollary}\label{main_cor}
 There exists an effective constant $N\in \N$, depending only on $h(\phi)$ and $p^e$, such that if $n\in \mathcal Z_s(\phi)$, then $n\leq N$. In particular, if $e=0$ then $N$ does not depend on $p$.
\end{corollary}
\begin{proof}
 Write $c_n=d_ny_n^2$, where $d_n,y_n\in \overline{\F}_p[t]$ and $d_n$ is squarefree. Assume that $n\in \mathcal Z_s(\phi)$. Then every irreducible factor of $d_n$ divides a non-zero $c_m$ for some $m<n$. Let $r$ be such an irreducible factor.
If $m>n/2$, then $r$ also divides $\phi^{(n-m)}(0)$, as the following computation shows:
 $$\phi^{(n-m)}(0)\equiv \phi^{(n-m)}(\phi^{(m)}(\gamma))\equiv \phi^{(n)}(\gamma)=c_n\equiv 0 \bmod r.$$
 It follows that $d_n=\prod_{i}r_i$, where $r_i\mid c_{m_i}$ or $r_i\mid \phi^{(m_i)}(0)$ for some $1\leq m_i\leq \floor{n/2}$ with $c_{m_i}\neq 0$. Therefore we necessarily have that $d_n$ divides $\displaystyle \prod^{\floor{n/2}}_{\substack{i=1\\c_i\neq 0}}c_i \prod^{\floor{n/2}}_{\substack{i=1\\\phi^{(i)}(0)\neq 0}}\phi^{(i)}(0)$. In turn,  we have the bound:
 $$h(d_n)\leq \sum_{i=1}^{\floor{n/2}}(h(c_i)+h(\phi^{(i)}(0))),$$
 while clearly if $\gamma=0$ then we have the better bound:
 $$h(d_n)\leq \sum_{i=1}^{\floor{n/2}}h(c_i).$$
 Now just use Theorems \ref{height_bound} and \ref{height_bound_sing}, together with Lemma \ref{height_lemma}, to get an inequality of the type $2^n\leq C2^{\floor{n/2}}+D$, for some constants $C,D$ depending only on $p^e$ and on constants which can be upper bounded by a function of $h(\phi)$, and deduce the claim.
\end{proof}

\begin{remark}
 Following the arguments in the proof of Corollary \ref{main_cor} and using Lemma \ref{height_lemma}, one can make the constant $N$ mentioned in Corollary \ref{main_cor} completely explicit. We will compute it for the polynomial $x^2+t$ in Section \ref{sec:conjecture}.
 
 In general if, $h(\gamma)\neq h(c_1)$ then the bound on $n$ found in Corollary \ref{main_cor} actually depends only on $p^e$. In fact, in this case Lemma \ref{height_lemma} and Theorem \ref{height_bound} imply that if $n\geq 3$ then:
 $$2^{n-2}h(\phi)\leq 8p^e(h(d_n))+8p^e(4+13h(\phi))+12h(\phi),$$
 and the arguments of Corollary \ref{main_cor} show that if $n\in\mathcal Z_s(\phi)$ then:
 \begin{equation}\label{index}
  2^{n-2}\leq 8p^e(3\cdot2^{\floor{n/2}}-3)+136p^e+12.
 \end{equation}
 Combined with the arguments in Section \ref{sec:reduction}, this yields a version of \cite[Theorem 1.1]{hindes1} in positive characteristic. Notice how the bound on the index of the representation is not absolute anymore, but depends on the dynamical inseparability degree of $\phi$. In particular, when $e=0$ one gets from \eqref{index} that $n\leq 12$. Since for any $\phi\in\O_{F,D}[t][x]$ there exist infinitely many primes $\p$ such that $\phi_\p$ has dynamical inseparability degree 0, observing that $\mathcal Z_s(\phi)\subseteq\mathcal Z_s(\phi_\p)$ for every $\p$ one can deduce that if $\phi$ is non-isotrivial and stable, then $|\log_2([\Aut(\T_\infty(\phi)):G_\infty(\phi)])|\leq 2^{12}-13$, improving \cite[Theorem 1.1]{hindes1} in the case where the ground field is a number field.
\end{remark}

\begin{proof}[Proof of Theorem \ref{main_thm}]

Let $\phi=(x-\gamma)^2-c_1\in \O_{F,D}[t][x]$ be non-isotrivial and with $c_1\neq 0$. Since $\phi$ is isotrivial if and only if $h(c_1+\gamma)=0$, the set $T$ of primes $\p$ of $\O_{F,D}$ such that $\phi_\p\in\F_\p[t][x]$ is isotrivial or $c_1\equiv 0 \bmod \p$ is finite. Moreover, it is clear from the definition that there exists a finite set $S\supseteq T$ of primes of $\O_{F,D}$ such that for all primes $\p\notin S$, the dynamical inseparability degree of $\phi_\p$ is zero and $h(\phi_\p)=h(\phi)$. It follows from Corollary \ref{main_cor} that there exists a constant $N$ such that, if $\p\notin S$, then $\max\{n\in\mathcal Z_s(\phi_\p)\}\leq N$. This leaves out the finitely many primes $\q\in S\setminus T$. For each of them, Corollary \ref{main_cor} yields a bound $N_\q$ on $\max\{n\in\mathcal Z_s(\phi_\q)\}$. Now clearly $N_\phi\coloneqq\max\{N,N_\q\colon \q\in S\setminus T\}$ is a bound on $\max\{n\in\mathcal Z_s(\phi_\p)\}$ for every prime $\p\notin T$. 
\end{proof}

\section{Reducing arboreal representations modulo primes}\label{sec:reduction}

This section is dedicated to proving Theorem \ref{gal_rep}, using Theorem \ref{main_thm}. In order to do so, we first need to recall some preliminary results.

Let $k$ be a field of characteristic $\neq 2$ and let $\phi\in k[t][x]$ be a monic, quadratic polynomial. Assume that $\phi^{(n)}$ is separable for every $n$. For every $n\geq 1$, denote by $K_n$ the splitting field of $\phi^{(n)}$ over $k^{\text{sep}}(t)$. Let $G_n^{\text{geo}}(\phi)\coloneqq\Gal(K_n/k^{\text{sep}}(t))$; recall that $G_\infty^{\text{geo}}(\phi)=\varprojlim_nG_n^{\text{geo}}(\phi)$.

It is easy to see that $[\Aut(\T_\infty(\phi)):G_\infty^{\text{geo}}(\phi)]<\infty$ if and only if there exists $N\in \N$ such that for every $n>N$ one has $[\Aut(\T_n(\phi)):G_n^{\text{geo}}(\phi)]=[\Aut(\T_N(\phi)):G_N^{\text{geo}}(\phi)]$.

For every $n\in \N$, we let $V_n(\phi)$ be the set of roots of $\phi^{(n)}$, i.e.\ the set of nodes of $\T_\infty(\phi)$ at distance $n$ from the root.

Finally, we denote by $S_2$ the symmetric group on two symbols $\{1,2\}$.

\begin{lemma}[{{\cite[Lemma 2.6]{jones3}}}]\label{discriminants}
 For every $n\geq 1$, let $\Delta_n$ be the discriminant of $\phi^{(n)}$. Then we have:
 $$\Delta_n=\pm 2^{2^{n+1}}\cdot\Delta_{n-1}^2\cdot c_n.$$
\end{lemma}

\begin{lemma}[{{\cite[Lemma 3.2]{jones3}}}]\label{maximal_degree}
 Let $n\in \N$ be such that $\phi^{(n)}$ is geometrically irreducible. Then the following three conditions are equivalent:
 \begin{enumerate}[i)]
  \item $[\Aut(\T_{n+1}(\phi)):G_{n+1}^{\text{geo}}(\phi)]=[\Aut(\T_{n}(\phi)):G_{n}^{\text{geo}}(\phi)]$;
  \item $G_{n+1}^{\text{geo}}(\phi)\cong G_n^{\text{geo}}(\phi)\wr S_2$;
  \item $c_{n+1}$ is not a square in $K_n$.
 \end{enumerate}
\end{lemma}

The wreath product mentioned in point $ii)$ of Lemma \ref{maximal_degree} is constructed in the following way: since $G_n^{\text{geo}}(\phi)$ acts on $V_n(\phi)$, it acts on $\prod_{v\in V_n(\phi)}S_2$ by permuting the indices. This induces a homomorphism $f\colon G_n^{\text{geo}}(\phi)\to \Aut\left(\prod_{v\in V_n(\phi)}S_2\right)$; we have $G_n^{\text{geo}}(\phi)\wr S_2\coloneqq G_n^{\text{geo}}(\phi)\ltimes_f\prod_{v\in V_n(\phi)}S_2$.

\begin{remark}\label{maximal_degree_remark}
 Notice that if $n\geq 1$ is such that $\phi^{(n)}$ is geometrically irreducible and one of the three conditions of Lemma \ref{maximal_degree} holds, then $\phi^{(n+1)}$ is geometrically irreducible as well. In fact, if $G_n^{\text{geo}}(\phi)$ acts transitively on $V_n(\phi)$, then $G_{n+1}^{\text{geo}}(\phi)\cong G_n^{\text{geo}}(\phi)\wr S_2$ acts transitively on $V_{n+1}(\phi)$, and hence $\phi^{(n+1)}$ is geometrically irreducible.
\end{remark}

\begin{corollary}\label{finite_index_cor}
Let $N\in \N$ be such that $\phi^{(N)}$ is geometrically irreducible. Suppose that $\max\{n\in\mathcal Z_s(\phi)\}\leq N$. Then the following hold.
\begin{enumerate}
 \item For every $m>N$, $[\Aut(\T_m(\phi))\colon G_m^{\text{geo}}(\phi)]=[\Aut(\T_N(\phi))\colon G_N^{\text{geo}}(\phi)]$.
 \item Let $k'$ be a field and $\psi \in k'[t][x]$ be a monic, quadratic polynomial such that:
 \begin{enumerate}[(a)]
  \item $(G_N^{\text{geo}}(\psi),\T_N(\psi))\cong_{\text{eq}}(G_N^{\text{geo}}(\phi),\T_N(\phi))$;
  \item for every $m>N$, $[\Aut(\T_m(\psi))\colon G_m^{\text{geo}}(\psi)]=[\Aut(\T_N(\psi))\colon G_N^{\text{geo}}(\psi)]$.
 \end{enumerate}
  Then $(G_\infty^{\text{geo}}(\psi),\T_\infty(\psi))\cong_{\text{eq}}(G_\infty^{\text{geo}}(\phi),\T_\infty(\phi))$.
\end{enumerate}
\end{corollary}
\begin{proof}(1) First, we claim that if $n\in \N$ is such that $\phi^{(n)}$ is geometrically irreducible and $n+1\notin\mathcal Z_s(\phi)$ (so that $c_{n+1}$ has a primitive squarefree divisor) then $[\Aut(\T_{n+1}(\phi))\colon G_{n+1}^{\text{geo}}(\phi)]=[\Aut(\T_n(\phi))\colon G_n^{\text{geo}}(\phi)]$. By Lemma \ref{maximal_degree}, to prove this it is enough to show that $c_{n+1}\notin K_n^2$. Suppose by contradiction that $c_{n+1}\in K_n^2$. Let $f\in k^{\text{sep}}[t]$ be a primitive, squarefree prime divisor of $c_{n+1}$, so that $c_{n+1}=f^{2r+1}\cdot g$, where $f,g\in k^{\text{sep}}[t]$ are coprime and $r\geq 0$. Since $c_{n+1}\in K_n^2$, in particular $c_{n+1}$ is a square in the integral closure $R$ of $k^{\text{sep}}[t]$ in $K_n$. It follows that the ideal generated by $c_{n+1}$ equals $\p_1^{2e_1}\cdot \ldots\cdot\p_s^{2e_s}$, where $\p_1,\ldots,\p_r$ are pairwise distinct primes of $R$ and $e_i\geq 1$ for every $i$. On the other hand, since $f^{2r+1}$ and $g$ are coprime, the ideals they generate are coprime, and thus, up to rearranging the $\p_i$'s, we can write $(f)^{2r+1}=(f^{2r+1})=\p_1^{2e_1}\cdot\ldots\cdot\p_{s'}^{2e_{s'}}$ for some $s'\leq s$. This of course implies that the ideal generated by $f$ is a square. But then $f$ ramifies in $R$, and now we have a contradiction: if $f$ ramifies in $R$ then $f$ should divide the discriminant of $\phi^{(n)}$  by Lemma \ref{discriminants}, but this cannot happen because $f$ is a primitive divisor of $c_{n+1}$.
 
 Now an easy induction using Remark \ref{maximal_degree_remark} concludes the proof of part (1).
 
 (2) In order to prove the claim, we will show that if for some $n\in \N$ the following three conditions hold:
  \begin{itemize}
   \item there exists an equivariant isomorphism 
   $$(\varphi_n,\iota_n)\colon (G_n^{\text{geo}}(\psi),\T_n(\psi))\to (G_n^{\text{geo}}(\phi),\T_n(\phi)),$$
   \item $G_{n+1}^{\text{geo}}(\psi)\cong G_{n}^{\text{geo}}(\psi)\wr S_2$,
   \item $G_{n+1}^{\text{geo}}(\phi)\cong G_{n}^{\text{geo}}(\phi)\wr S_2$,
  \end{itemize}
 then there exists an equivariant isomorphism
 $$(\varphi_{n+1},\iota_{n+1})\colon(G_{n+1}^{\text{geo}}(\psi),\T_{n+1}(\psi))\to (G_{n+1}^{\text{geo}}(\phi),\T_{n+1}(\phi))$$ that extends $(\varphi_n,\iota_n)$. This fact, together with part (1), allows to construct inductively an equivariant isomorphism:
 $$(\varphi_\infty,\iota_\infty)\colon (G_\infty^{\text{geo}}(\psi),\T_\infty(\psi))\to (G_\infty^{\text{geo}}(\phi),\T_\infty(\phi))$$
 in the obvious way: if $g=(g_1,\ldots,g_n,\ldots)\in G_\infty^{\text{geo}}(\psi)$ and $v\in V_m(\psi)$ for some $m\in \N$, then $(\varphi_\infty,\iota_\infty)(g,v)=((\varphi_1(g_1),\ldots,\varphi_n(g_n),\ldots),\iota_m(v))$, concluding the proof of part (2).
 
 So let $(\varphi_n,\iota_n)\colon (G_n^{\text{geo}}(\psi),\T_n(\psi))\to (G_n^{\text{geo}}(\phi),\T_n(\phi))$ be an equivariant isomorphism. Let $(g,(\varepsilon_v)_{v\in V_n(\psi)})$ be an element of $G_{n+1}^{\text{geo}}(\psi)\cong G_n^{\text{geo}}(\psi)\wr S_2$. Here $g\in G_n^{\text{geo}}(\psi)$, while $\varepsilon_v\in S_2$ for every $v\in V_n(\psi)$. Now define the following map:
 $$\varphi_{n+1}\colon G_{n+1}^{\text{geo}}(\psi)\to G_{n+1}^{\text{geo}}(\phi)$$
 $$(g,(\varepsilon_v)_{v\in V_n(\psi)})\mapsto (\varphi_n(g),(\varepsilon_{\iota_n^{-1}(w)})_{w\in V_n(\phi)});$$
 one readily checks that this is a group isomorphism extending $\varphi_n$. In order to extend $\iota_n$, first notice that every element of $V_{n+1}(\psi)$ can be represented in a unique way as a pair $(v,s_v)$, where $v\in V_n(\psi)$ and $s_v$ is a map $\{v\}\to\{1,2\}$. With this notation, if $h=(g,(\varepsilon_v)_{v\in V_n(\psi)})\in G_{n+1}^{\text{geo}}(\psi)$ and $(w,s_w)\in V_{n+1}(\psi)$, then ${}^{h}(w,s_w)=({}^{g}w,\varepsilon_{{}^{g}w}\circ s_w)$. Now define
 $$\iota_{n+1}\colon V_{n+1}(\psi)\to V_{n+1}(\phi)$$
 $$(v,s_v)\mapsto (\iota_n(v),t_{\iota_n(v)}),$$
 where $t_{\iota_n(v)}$ is the composition of the maps $s_v$ and $\iota_n^{-1}\colon\{\iota_n(v)\}\to\{v\}$. It is now just a matter of checking that the pair $(\varphi_{n+1},\iota_{n+1})$ is an equivariant isomorphism.\end{proof}

\begin{lemma}\label{galois_action}
 Let $f\in \O_{F,D}[t][x]$ and let $G\coloneqq \Gal(f/\overline{F}(t))$. For every prime $\p$ of $\O_{F,D}$ let $f_\p\in\F_\p[t][x]$ be the reduced polynomial, and let $G_\p\coloneqq \Gal(f_\p/\overline{\F}_\p(t))$. Let $R_f$ be the set of roots of $f$ in $\overline{F(t)}$ and $R_{f_\p}$ the set of roots of $f_\p$ in $\overline{\F_\p(t)}$. Then there exists a finite, effective set $T_f$ of primes of $\O_{F,D}$ with the following property: there exists an equivariant isomorphism $(G_\p,R_{f_\p})\to (G,R_f)$ if and only if $\p\notin T_f$ .
\end{lemma}
\begin{proof}
 This is a well-known fact, let us sketch a proof for the sake of completeness.
 
 Start by fixing a bijection $\iota\colon R_f\to\{1,\ldots,n\}$. This realizes $G$ as a subgroup of the symmetric group $S_n$. If $\p$ is a prime of $\O_{F,D}$ for which $f_\p$ has $n$ roots over $\overline{\F_\p(t)}$ (and this happens for all primes but at most finitely many), then \cite[Theorem 5.8.5]{chambert} ensures the existence of a bijection $\iota_\p\colon R_{f_\p}\to\{1,\ldots,n\}$ that induces a monomorphism $\varphi_\p\colon G_\p\hookrightarrow G$ with the property that $(\iota^{-1}\circ\iota_\p)({}^{g}\alpha)={}^{\varphi_\p(g)}(\iota^{-1}\circ\iota)(\alpha)$ for every $g\in G_\p$ and every $\alpha\in R_{f_\p}$. It remains to show that for all but finitely many $\p$, $\varphi_\p$ is an isomorphism. To see this, for every subgroup $H\subseteq G\subseteq S_n$ write a squarefree resolvent $r_{f,H}(y,X_1,\ldots,X_n)\in \O_{F,D}[y,X_1,\ldots,X_n]$ for the pair $(f,H)$ (possibly enlarging $D$ by a finite set of primes of $\O_F$). This is a polynomial with the following properties:
 \begin{itemize}
  \item the coefficients of $r_{f,H}$, when the latter is thought as a polynomial in $y$ with coefficients in $\O_{F,D}[X_1,\ldots,X_n]$, are symmetric functions in the $X_i$'s;
  \item if $\{\alpha_1,\ldots,\alpha_n\}$ are the roots of $f$ then the discriminant of $r_{f,H}(y,\alpha_1,\ldots,\alpha_n)$ is a non-zero element of $\O_{F,D}[t]$;
  \item the Galois group of $f$ is contained in a conjugate of $H$ in $G$ if and only if $r_{f,H}(y,\alpha_1,\ldots,\alpha_n)$ has a root in $\overline{F}(t)$.
 \end{itemize}
  See \citep{arnaudies,cangelmi,cohen} for more on the theory of resolvents. Note that since the coefficients of $g$ are symmetric functions in the $X_i$'s, then the coefficients of $r_{f,H}(y,\alpha_1,\ldots,\alpha_n)$ can be written as polynomials in the coefficients of $f$. It is then clear that for all but finitely many primes $\p$ the reduction of $r_{f,H}$ modulo $\p$ is a squarefree resolvent for the pair $(f_\p,H)$. Now, since of course $G\not\subseteq H$ for every proper subgroup $H$, this means that the resolvent $r_{f,H}$ has no roots in $\overline{F}(t)$ for every $H\subsetneq G$. To conclude the proof, it is then enough to invoke Bertini-Noether Theorem (see \cite[Proposition 8.5.7]{lang}): a polynomial in $\O_{F,D}[t,x]$ that is irreducible in $\overline{F}[t,x]$ is irreducible in $\overline{\F}_\p[t,x]$ for all but finitely many primes $\p$. In our setting, this implies that for all but finitely many $\p$, the resolvent $r_{f_\p,H}$ has no roots in $\overline{\F}_\p$, and hence $\varphi_\p(G_\p)\not\subseteq H$ for all $H$, i.e.\ $\varphi_\p$ is an isomorphism.
\end{proof}

With all the information above, we can now prove Theorem \ref{gal_rep}.

\begin{proof}[Proof of Theorem \ref{gal_rep}]
 Recall that here $\phi=(x-\gamma)^2-c_1\in\O_{F,D}[t][x]$ is non-isotrivial and geometrically stable. Notice that this implies, in particular, that $c_1\neq 0$ and therefore we can apply Theorem \ref{main_thm} to $\phi$. Let $N_\phi$ be the constant determined by Theorem \ref{main_thm}. Now let $T_\phi$ be the finite set of primes determined by Lemma \ref{galois_action} for $f=\phi^{(N_\phi)}$ and let $S_\phi$ be the finite set of primes $\p$ of $\O_{F,D}$ such that $\phi_\p$ is isotrivial or $\p\in T_\phi$. Clearly, if $\p\in S_\phi$, then $(G_\infty^{\text{geo}}(\phi_\p),\T_{\infty}(\phi_\p))\not\cong_{\text{eq}}(G_\infty^{\text{geo}}(\phi),\T_{\infty}(\phi))$: if $\p\in T_\phi$ this is obvious, because $(G_{N_{\phi}}^{\text{geo}}(\phi_\p),\T_{N_\phi}(\phi_\p))\not\cong_{\text{eq}}(G_{N_\phi}^{\text{geo}}(\phi),\T_{N_\phi}(\phi))$; if $\phi_\p$ is isotrivial then $[\Aut(\T_\infty(\phi_\p)):G_\infty^{\text{geo}}(\phi_\p)]=\infty$, because $\phi_\p$ is post-critically finite (see \cite[Theorem 3.1]{jones2}), while $[\Aut(\T_\infty(\phi)):G_\infty^{\text{geo}}(\phi)]<\infty$ by \cite[Theorem 1]{hindes1}.
 
 Conversely, suppose that $\p\notin S_\phi$. Since $(G_{N_{\phi}}^{\text{geo}}(\phi_\p),\T_{N_\phi}(\phi_\p))\cong_{\text{eq}}(G_{N_\phi}^{\text{geo}}(\phi),\T_{N_\phi}(\phi))$ and $\phi^{(N_\phi)}$ is geometrically irreducible, then $\phi_\p^{(N_\phi)}$ is geometrically irreducible as well. This implies in particular that $c_1\neq 0\bmod \p$, and therefore we have that $\max\{n\in\mathcal Z_s(\phi_\p)\}\leq N_\phi$ by Theorem \ref{main_thm}. Now simply observe that for every prime $\p\notin S_\phi$ we have that $\mathcal Z_s(\phi)\subseteq \mathcal Z_s(\phi_\p)$, and therefore $\max\{n\in \mathcal Z_s(\phi)\}\leq N_\phi$. Hence, we are in the exact position to apply Corollary \ref{finite_index_cor} and conclude the proof.
\end{proof}

\section{Theorem \ref{gal_rep} at work: a conjecture of Jones}\label{sec:conjecture}

The goal of this section is to show how Theorem \ref{gal_rep} can be applied to the polynomial $\phi=x^2+t$ to prove Theorem \ref{jones_conjecture}. In fact, we will prove the following theorem.

\begin{theorem}\label{conjecture_proof}
 Let $\phi=x^2+t\in \Z[t][x]$. Then the arboreal representation $\rho_{\phi_p}$ is geometrically surjective for all odd primes $p$.
\end{theorem}
Before proving Theorem \ref{conjecture_proof}, let us recall the following criterion for surjectivity, whose idea is due to Stoll.
\begin{theorem}[{{\cite{hindes2,stoll}}}]\label{stoll_thm}
 Assume that $\car k\neq 2$ and let $\phi\in k[x]$ be a monic, quadratic polynomial such that all iterates are separable. Let $\{c_n\}$ be its adjusted post-critical orbit. Then $\rho_\phi$ is surjective if and only if for every $n\geq 1$ the $\F_2$-vector space generated by $c_1,\ldots,c_n$ in $k^{\times}/{k^{\times}}^2$ has dimension $n$.
\end{theorem}

Next, let us remark that Theorem \ref{conjecture_proof} implies Theorem \ref{jones_conjecture}. In fact, as noticed by Jones in \cite[p.\ 1125]{jones1}, if $\car k=0$ then $\rho_\phi$ is geometrically surjective. On the other hand, let $\car k=p$. By Theorem \ref{stoll_thm} the arboreal representation $\rho_\phi$ is surjective if and only if there is no product of the $c_i$'s that is a square in $k(t)$. But an algebraic closure of $k$ contains an algebraic closure of $\F_p$, and therefore, since $c_n\in \F_p[t]$ for all $n$, if no product of the $c_i$'s is a square in $\overline{\F}_p(t)$, then no product of the $c_i$'s can be a square in $k(t)$.

\begin{proof}[Proof of Theorem \ref{conjecture_proof}]
 Clearly, for every odd prime $p$, we have that $\phi_p$ is non-isotrivial, $c_1\not\equiv 0 \bmod p$ and the dynamical inseparability degree of $\phi_p$ is 0. From Theorem \ref{height_bound} it follows immediately that if $n\geq 3$ then:
 \begin{equation}\label{deg_bound}
  \deg c_{n-1,p}\leq 8\deg d_{n,p}+144,
 \end{equation}
 where the subscript $p$ denotes the reduction modulo $p$. Clearly we have that $\deg c_{n-1,p}=2^{n-2}$ for every $n\geq 2$ and, as explained in the proof of Corollary \ref{main_cor}, if $n\in\mathcal Z_s(\phi_p)$ then $\deg d_{n,p}\leq\sum_{i=1}^{\floor{n/2}}2^{i-1}=2^{\floor{n/2}}-1$, which in turn implies, together with \eqref{deg_bound}, that the constant $N_\phi$ given by Theorem \ref{main_thm} is 11.
 
 As we mentioned above, it is well-known that $\rho_\phi$ is geometrically surjective in characteristic zero, and so in particular over $\Q$. This implies that $\phi$ is geometrically stable, and hence we can apply Theorem \ref{gal_rep} to $\phi$. As explained in the proof of the theorem, we have that $G^{\text{geo}}_\infty(\phi_p)\cong G^{\text{geo}}_\infty(\phi)\cong\Aut(\T_\infty(\phi))$ for all primes $p$ but possibily the finitely many ones such that $(G_{11}^{\text{geo}}(\phi_p),\T_{11}(\phi_p))\not\cong_{\text{eq}}(G_{11}^{\text{geo}}(\phi),\T_{11}(\phi))$. Since $G_{11}^{\text{geo}}(\phi)=\Aut(\T_{11}(\phi))$, it is easy to see, using the argument in Corollary \ref{finite_index_cor}, that the two pairs are equivariantly isomorphic if and only if $G_{11}^{\text{geo}}(\phi_p)$ and $G_{11}^{\text{geo}}(\phi)$ are isomorphic as groups.
 
 Thus, it is enough to find the finitely many primes for which, possibly, one has that $G_{11}^{\text{geo}}(\phi_p)\neq \Aut(\T_{11}(\phi_p))$. In order to find them, we use Theorem \ref{stoll_thm}. This tells us that if $G_{11}^{\text{geo}}(\phi_p)$ is smaller than $\Aut(\T_{11}(\phi_p))$ then there exists a non-empty subset $I\subseteq\{1,\ldots,11\}$ such that $\prod_{i\in I}c_{i,p}$ is a square in $\overline{\F}_p[t]$. Since $\prod_{i\in I}c_i\in \Z[t]$ is not a square in $\overline{\Q}[t]$, its squarefree part $d_I$ has positive degree. On the other hand, if $\prod_{i\in I}c_{i,p}$ is a square in $\overline{\F}_p[t]$ then one of the following holds: $\prod_{i\in I}c_{i,p}=0$ or $d_I$ is a square modulo $p$, and hence $p\mid \disc (d_I)$. Clearly we cannot have $\prod_{i\in I}c_{i,p}=0$ since all $c_i$'s are monic. Thus, to conclude the proof it is enough to check that for all plausible non-empty subsets $I\subseteq \{1,\ldots,11\}$ and all the finitely many odd primes dividing $\disc(d_I)$, the product $\prod_{i\in I}c_{i,p}$ is not a square in $\overline{\F}_p[t]$. This has been verified\footnote{The code is available upon request.} with SAGE \cite{sagemath}.
\end{proof}

\bibliographystyle{plain}
\bibliography{bibliography}

\end{document}